\documentclass[reqno]{amsart}
\usepackage{amsfonts}

\newtheorem{theorem}{Theorem}
\theoremstyle{plain}

\newtheorem{corollary}{Corollary}

\newtheorem{lemma}{Lemma}

\newtheorem{proposition}{Proposition}
\newtheorem{remark}{Remark}

\numberwithin{equation}{section}
\numberwithin{theorem}{section}
\numberwithin{algorithm}{section}
\numberwithin{axiom}{section}
\numberwithin{case}{section}
\numberwithin{claim}{section}
\numberwithin{conclusion}{section}
\numberwithin{condition}{section}
\numberwithin{conjecture}{section}
\numberwithin{corollary}{section}
\numberwithin{criterion}{section}
\numberwithin{definition}{section}
\numberwithin{example}{section}
\numberwithin{exercise}{section}
\numberwithin{lemma}{section}
\numberwithin{notation}{section}
\numberwithin{problem}{section}
\numberwithin{proposition}{section}
\numberwithin{remark}{section}
\numberwithin{solution}{section}

\input{tcilatex}

\begin{document}
\title[Sign of Green's function and the $Q$ curvature]{Sign of Green's
function of Paneitz operators and the $Q$ curvature}
\author{Fengbo Hang}
\address{Courant Institute, New York University, 251 Mercer Street, New York
NY 10012}
\email{fengbo@cims.nyu.edu}
\author{Paul C. Yang}
\address{Department of Mathematics, Princeton University, Fine Hall,
Washington Road, Princeton NJ 08544}
\email{yang@math.princeton.edu}
\date{}

\begin{abstract}
In a conformal class of metrics with positive Yamabe invariant, we derive a
necessary and sufficient condition for the existence of metrics with
positive $Q$ curvature. The condition is conformally invariant. We also
prove some inequalities between the Green's functions of the conformal
Laplacian operator and the Paneitz operator.
\end{abstract}

\maketitle

\section{Introduction\label{sec1}}

Since the fundamental work \cite{CGY} in dimension $4$, the Paneitz operator
and associated $Q$ curvature in dimension other than $4$ (see \cite{B,P})
attracts much attention (see \cite{DHL,GM,HY1,HeR1,HeR2,HuR,QR} etc and the
references therein). Let $\left( M,g\right) $ be a smooth compact $n$
dimensional Riemannian manifold. For $n\geq 3$, the $Q$ curvature is given by%
\begin{equation}
Q=-\frac{1}{2\left( n-1\right) }\Delta R-\frac{2}{\left( n-2\right) ^{2}}%
\left\vert Rc\right\vert ^{2}+\frac{n^{3}-4n^{2}+16n-16}{8\left( n-1\right)
^{2}\left( n-2\right) ^{2}}R^{2}.  \label{eq1.1}
\end{equation}%
Here $R$ is the scalar curvature, $Rc$ is the Ricci tensor. The Paneitz
operator is given by%
\begin{eqnarray}
&&P\varphi  \label{eq1.2} \\
&=&\Delta ^{2}\varphi +\frac{4}{n-2}\func{div}\left( Rc\left( \nabla \varphi
,e_{i}\right) e_{i}\right) -\frac{n^{2}-4n+8}{2\left( n-1\right) \left(
n-2\right) }\func{div}\left( R\nabla \varphi \right) +\frac{n-4}{2}Q\varphi .
\notag
\end{eqnarray}%
Here $e_{1},\cdots ,e_{n}$ is a local orthonormal frame with respect to $g$.
For $n\neq 4$, under conformal transformation of the metric, the operator
satisfies%
\begin{equation}
P_{\rho ^{\frac{4}{n-4}}g}\varphi =\rho ^{-\frac{n+4}{n-4}}P_{g}\left( \rho
\varphi \right) .  \label{eq1.3}
\end{equation}%
Note this is similar to the conformal Laplacian operator, which appears
naturally when considering transformation law of the scalar curvature under
conformal change of metric (\cite{LP}). As a consequence we know%
\begin{equation}
P_{\rho ^{\frac{4}{n-4}}g}\varphi \cdot \psi d\mu _{\rho ^{\frac{4}{n-4}%
}g}=P_{g}\left( \rho \varphi \right) \cdot \rho \psi d\mu _{g}.
\label{eq1.4}
\end{equation}%
Here $\mu _{g}$ is the measure associated with metric $g$. Moreover%
\begin{equation}
\ker P_{g}=0\Leftrightarrow \ker P_{\rho ^{\frac{4}{n-4}}g}=0,  \label{eq1.5}
\end{equation}%
and under this assumption, the Green's functions $G_{P}$ satisfy the
transformation law%
\begin{equation}
G_{P,\rho ^{\frac{4}{n-4}}g}\left( p,q\right) =\rho \left( p\right)
^{-1}\rho \left( q\right) ^{-1}G_{P,g}\left( p,q\right) .  \label{eq1.6}
\end{equation}%
For $u,v\in C^{\infty }\left( M\right) $, we denote the quadratic form
associated with $P$ as%
\begin{eqnarray}
&&E\left( u,v\right)  \label{eq1.7} \\
&=&\int_{M}Pu\cdot vd\mu  \notag \\
&=&\int_{M}\left( \Delta u\Delta v-\frac{4}{n-2}Rc\left( \nabla u,\nabla
v\right) +\frac{n^{2}-4n+8}{2\left( n-1\right) \left( n-2\right) }R\nabla
u\cdot \nabla v\right.  \notag \\
&&\left. +\frac{n-4}{2}Quv\right) d\mu  \notag
\end{eqnarray}%
and%
\begin{equation}
E\left( u\right) =E\left( u,u\right) .  \label{eq1.8}
\end{equation}%
By the integration by parts formula in (\ref{eq1.7}) we know $E\left(
u,v\right) $ also makes sense for $u,v\in H^{2}\left( M\right) $.

To continue we recall (see \cite{LP}) for $n\geq 3$, on a smooth compact
Riemannian manifold $\left( M^{n},g\right) $, the conformal Laplacian
operator is given by%
\begin{equation}
L_{g}\varphi =-\frac{4\left( n-1\right) }{n-2}\Delta \varphi +R\varphi .
\label{eq1.9}
\end{equation}%
The Yamabe invariant is defined as%
\begin{eqnarray}
&&Y\left( g\right)  \label{eq1.10} \\
&=&\inf \left\{ \frac{\int_{M}\widetilde{R}d\widetilde{\mu }}{\left( 
\widetilde{\mu }\left( M\right) \right) ^{\frac{n-2}{n}}}:\widetilde{g}=\rho
^{2}g\text{ for some positive smooth function }\rho \right\}  \notag \\
&=&\inf \left\{ \frac{\int_{M}L_{g}\varphi \cdot \varphi d\mu }{\left\Vert
\varphi \right\Vert _{L^{\frac{2n}{n-2}}}^{2}}:\varphi \text{ is a nonzero
smooth function on }M\right\} .  \notag
\end{eqnarray}%
A basic but useful fact about the scalar curvature is%
\begin{gather}
Y\left( g\right) >0\Leftrightarrow \lambda _{1}\left( L_{g}\right) >0
\label{eq1.11} \\
\Leftrightarrow g\text{ is conformal to a metric with scalar curvature}>0%
\text{.}  \notag
\end{gather}%
Indeed more is true, namely the equivalence still holds if we replace all "$%
> $" by "$="$ or "$<$". Here $\lambda _{1}\left( L_{g}\right) $ is the first
eigenvalue of $L_{g}$. It is clear $Y\left( g\right) $ is a conformal
invariant, on the other hand the sign of $\lambda _{1}\left( L_{g}\right) $
is also conformally invariant. The main reason that (\ref{eq1.11}) holds is
based on the fact the first eigenfunction of a second order symmetric
differential operator does not change sign. Unfortunately such kind of
property is known to be not valid for higher order operators. The following
question keeps puzzling people from the beginning of research on $Q$
curvature in dimension other than $4$, namely: can we find a conformal
invariant condition which is equivalent to the existence of positive $Q$
curvature in the conformal class (in the same spirit as (\ref{eq1.11}))?
Here we give an answer to this question for conformal class with positive
Yamabe invariant.

\begin{theorem}
\label{thm1.1}Let $n>4$ and $\left( M^{n},g\right) $ be a smooth compact
Riemannian manifold with Yamabe invariant $Y\left( g\right) >0$, then the
following statements are equivalent

\begin{enumerate}
\item there exists a positive smooth function $\rho $ with $Q_{\rho ^{2}g}>0$%
.

\item $\ker P_{g}=0$ and the Green's function $G_{P}\left( p,q\right) >0$
for any $p,q\in M,p\neq q$.

\item $\ker P_{g}=0$ and there exists a $p\in M$ such that $G_{P}\left(
p,q\right) >0$ for $q\in M\backslash \left\{ p\right\} $.
\end{enumerate}
\end{theorem}

Along the way we also find the following comparison inequality between
Green's function of $L$ and $P$.

\begin{proposition}
\label{prop1.1}Assume $n>4$, $\left( M^{n},g\right) $ is a smooth compact
Riemannian manifold with $Y\left( g\right) >0$, $Q\geq 0$ and not
identically zero, then $\ker P=0$ and%
\begin{equation}
G_{L}^{\frac{n-4}{n-2}}\leq c_{n}G_{P}.  \label{eq1.12}
\end{equation}%
Here%
\begin{equation}
c_{n}=2^{-\frac{n-6}{n-2}}n^{\frac{2}{n-2}}\left( n-1\right) ^{-\frac{n-4}{%
n-2}}\left( n-2\right) \left( n-4\right) \omega _{n}^{\frac{2}{n-2}},
\label{eq1.13}
\end{equation}%
$\omega _{n}$ is the volume of unit ball in $\mathbb{R}^{n}$. Moreover if $%
G_{L}^{\frac{n-4}{n-2}}\left( p,q\right) =c_{n}G_{P}\left( p,q\right) $ for
some $p\neq q$, then $\left( M,g\right) $ is conformal diffeomorphic to the
standard sphere.
\end{proposition}

In dimension $3$ we have

\begin{theorem}
\label{thm1.2}Let $\left( M,g\right) $ be a smooth compact $3$ dimensional
Riemannian manifold with Yamabe invariant $Y\left( g\right) >0$, then the
following statements are equivalent

\begin{enumerate}
\item there exists a positive smooth function $\rho $ with $Q_{\rho ^{2}g}>0$%
.

\item $\ker P_{g}=0$ and $G_{P}\left( p,q\right) <0$ for any $p,q\in M,p\neq
q$.

\item $\ker P_{g}=0$ and there exists a $p\in M$ such that $G_{P}\left(
p,q\right) <0$ for $q\in M\backslash \left\{ p\right\} $.
\end{enumerate}
\end{theorem}

Similar to Proposition \ref{prop1.1}, we have

\begin{proposition}
\label{prop1.2}Let $\left( M,g\right) $ be a smooth compact $3$ dimensional
Riemannian manifold with $Y\left( g\right) >0$, $Q\geq 0$ and not
identically zero, then $\ker P=0$ and%
\begin{equation}
G_{L}^{-1}\leq -256\pi ^{2}G_{P}.  \label{eq1.14}
\end{equation}%
If for some $p,q\in M$, $G_{L}^{-1}\left( p,q\right) =-256\pi
^{2}G_{P}\left( p,q\right) $, then $\left( M,g\right) $ is conformal
diffeomorphic to the standard $S^{3}$ (note here $p$ can be equal to $q$).
\end{proposition}

In dimension $4$ we have the following (see Corollary \ref{cor5.1})

\begin{proposition}
\label{prop1.3}Assume $\left( M,g\right) $ is a smooth compact $4$
dimensional Riemannian manifold, $Y\left( g\right) >0$, then for any $p\in M$%
,%
\begin{equation}
\int_{M}Qd\mu +\frac{1}{2}\int_{M}\left\vert Rc_{G_{L,p}^{2}g}\right\vert
_{G_{L,p}^{2}g}^{2}d\mu _{G_{L,p}^{2}g}=16\pi ^{2}.  \label{eq1.15}
\end{equation}%
In particular, $\int_{M}Qd\mu \leq 16\pi ^{2}$ and equality holds if and
only if $\left( M,g\right) $ is conformal diffeomorphic to the standard $%
S^{4}$.
\end{proposition}

It is worthwhile to point out that the proof of Theorem B in \cite{G}, which
gives the inequality in Proposition \ref{prop1.3}, is elementary and does
not use the positive mass theorem. Our argument is also elementary and
identifies the difference between $\int_{M}Qd\mu $ and $16\pi ^{2}$.

Theorem \ref{thm1.1} and \ref{thm1.2} are motivated by works on the $Q$
curvature in dimension $5$ or higher (\cite{GM,HeR1,HeR2,HuR}) and in
dimension $3$ (\cite{HY1,HY2,HY3}). In \cite{HeR1,HeR2}, it was shown in
some cases compactness property for solutions of the $Q$ curvature equation
can be derived under the assumption that the Green's function is positive.
Recently \cite{GM} showed that the Green's function is indeed positive when
both scalar curvature and $Q$ curvature are positive. Theorem \ref{thm1.1}
says we could relax the assumption to $Y\left( g\right) >0,Q_{g}>0$. Whether
these two kinds of assumptions are equivalent or not is still unknown. The
main approach in \cite{GM} is roughly speaking by applying the maximum
principle twice. In \cite{HY3}, by replacing maximum principle with the weak
Harnack inequality it was shown that for metrics with $R>0$ and $Q>0$, $P$
is invertible and $G_{P}\left( p,q\right) <0$ for $p\neq q$. Theorem \ref%
{thm1.2} relax the assumption to $Y\left( g\right) >0$ and $Q>0$. The main
new ingredient in our proof of Theorem \ref{thm1.1} and \ref{thm1.2} is the
formula (\ref{eq2.1}), which is closely related to the arguments in \cite%
{HuR}. In \cite{HY4}, we will apply the results on Green's function to
solution of $Q$ curvature equations. In section \ref{sec2} we will prove the
main formula (\ref{eq2.1}). In sections \ref{sec3} and \ref{sec4} we will
prove Theorem \ref{thm1.1} and \ref{thm1.2} respectively. In section \ref%
{sec5} we will derive the corresponding formula of (\ref{eq2.1}) in
dimension $4$. In particular Proposition \ref{prop1.3} follows from the
formula. In section \ref{sec6}, we will show that the positive mass theorem
for Paneitz operator in \cite{GM,HuR} can be deduced from (\ref{eq2.1}) too.

\section{An identity connecting the Green's function of conformal Laplacian
operator and Paneitz operator\label{sec2}}

Here we will derive an interesting formula which illustrates the close
relation between Green's function of conformal Laplacian operator and the
Paneitz operator. This identity will play a crucial role later.

To motivate the discussion, we note that positive Yamabe invariant implies
we have a positive Green's function for the conformal Laplacian operator.
Even though we do not know whether $P$ is invertible or not, we may still
try to search for its Green's function. Note that the possible Green's
function should have same highest order singular term as $G_{L,p}^{\frac{n-4%
}{n-2}}$ (modulus dimension constant), we can use $G_{L,p}^{\frac{n-4}{n-2}}$
as a first step approximation. Along this line we compute $P\left( G_{L,p}^{%
\frac{n-4}{n-2}}\right) $ and arrive at the interesting formula (\ref{eq2.1}%
).

\begin{proposition}
\label{prop2.1}Assume $n\geq 3$, $n\neq 4$, $\left( M,g\right) $ is a $n$
dimensional smooth compact Riemannian manifold with $Y\left( g\right) >0$, $%
p\in M$, then we have $G_{L,p}^{\frac{n-4}{n-2}}\left\vert Rc_{G_{L,p}^{%
\frac{4}{n-2}}g}\right\vert _{g}^{2}\in L^{1}\left( M\right) $ and%
\begin{equation}
P\left( G_{L,p}^{\frac{n-4}{n-2}}\right) =c_{n}\delta _{p}-\frac{n-4}{\left(
n-2\right) ^{2}}G_{L,p}^{\frac{n-4}{n-2}}\left\vert Rc_{G_{L,p}^{\frac{4}{n-2%
}}g}\right\vert _{g}^{2}  \label{eq2.1}
\end{equation}%
in distribution sense. Here%
\begin{equation}
c_{n}=2^{-\frac{n-6}{n-2}}n^{\frac{2}{n-2}}\left( n-1\right) ^{-\frac{n-4}{%
n-2}}\left( n-2\right) \left( n-4\right) \omega _{n}^{\frac{2}{n-2}},
\label{eq2.2}
\end{equation}%
$\omega _{n}$ is the volume of unit ball in $\mathbb{R}^{n}$, $G_{L,p}$ is
the Green's function of conformal Laplacian operator $L=-\frac{4\left(
n-1\right) }{n-2}\Delta +R$ with pole at $p$.
\end{proposition}

It is worth pointing out that the metric $G_{L,p}^{\frac{4}{n-2}}g$ on $%
M\backslash \left\{ p\right\} $ is exactly the stereographic projection of $%
\left( M,g\right) $ at $p$ (\cite{LP}). To prove the proposition, let us
first check what happens under a conformal change of the metric. If $\rho
\in C^{\infty }\left( M\right) $ is a positive function, let $\widetilde{g}%
=\rho ^{\frac{4}{n-2}}g$, then using%
\begin{equation*}
G_{\widetilde{L},p}\left( q\right) =\rho \left( p\right) ^{-1}\rho \left(
q\right) ^{-1}G_{L,p}\left( q\right)
\end{equation*}%
we see%
\begin{equation}
G_{\widetilde{L},p}^{\frac{n-4}{n-2}}\left\vert Rc_{G_{\widetilde{L},p}^{%
\frac{4}{n-2}}\widetilde{g}}\right\vert _{\widetilde{g}}^{2}d\widetilde{\mu }%
=\rho \left( p\right) ^{-\frac{n-4}{n-2}}\rho ^{\frac{n-4}{n-2}}G_{L,p}^{%
\frac{n-4}{n-2}}\left\vert Rc_{G_{L,p}^{\frac{4}{n-2}}g}\right\vert
_{g}^{2}d\mu .  \label{eq2.3}
\end{equation}%
Hence we only need to check $G_{L,p}^{\frac{n-4}{n-2}}\left\vert
Rc_{G_{L,p}^{\frac{4}{n-2}}g}\right\vert _{g}^{2}\in L^{1}\left( M\right) $
for a conformal metric.

By the existence of conformal normal coordinate (\cite{LP}) we can assume $%
\exp _{p}$ preserve the volume near $p$. Let $x_{1},\cdots ,x_{n}$ be a
normal coordinate at $p$, denote $r=\left\vert x\right\vert $, then (see 
\cite{LP})%
\begin{equation}
G_{L,p}=\frac{1}{4n\left( n-1\right) \omega _{n}}r^{2-n}\left( 1+O^{\left(
4\right) }\left( r\right) \right) .  \label{eq2.4}
\end{equation}%
As usual, we say $f=O^{\left( m\right) }\left( r^{\theta }\right) $ to mean $%
f$ is $C^{m}$ in the punctured neighborhood with $\partial _{i_{1}\cdots
i_{k}}f=O\left( r^{\theta -k}\right) $ for $0\leq k\leq m$. By (\ref{eq2.4})
and the transformation law%
\begin{eqnarray}
Rc_{G_{L,p}^{\frac{4}{n-2}}g} &=&Rc-2D^{2}\log G_{L,p}+\frac{4}{n-2}d\log
G_{L,p}\otimes d\log G_{L,p}  \label{eq2.5} \\
&&-\left( \frac{2}{n-2}\Delta \log G_{L,p}+\frac{4}{n-2}\left\vert \nabla
\log G_{L,p}\right\vert ^{2}\right) g,  \notag
\end{eqnarray}%
careful calculation shows%
\begin{equation}
\left\vert Rc_{G_{L,p}^{\frac{4}{n-2}}g}\right\vert _{g}=O\left( \frac{1}{r}%
\right) .  \label{eq2.6}
\end{equation}%
It follows that%
\begin{equation*}
G_{L,p}^{\frac{n-4}{n-2}}\left\vert Rc_{G_{L,p}^{\frac{4}{n-2}}g}\right\vert
_{g}^{2}=O\left( r^{2-n}\right)
\end{equation*}%
hence it belongs to $L^{1}\left( M\right) $.

To continue, we observe that equation (\ref{eq2.1}) is the same as%
\begin{equation}
\int_{M}G_{L,p}^{\frac{n-4}{n-2}}P\varphi d\mu =c_{n}\varphi \left( p\right)
-\frac{n-4}{\left( n-2\right) ^{2}}\int_{M}G_{L,p}^{\frac{n-4}{n-2}%
}\left\vert Rc_{G_{L,p}^{\frac{4}{n-2}}g}\right\vert _{g}^{2}\varphi d\mu
\label{eq2.7}
\end{equation}%
for any $\varphi \in C^{\infty }\left( M\right) $. A similar check as before
shows (\ref{eq2.7}) is conformally invariant.

Again we assume $\exp _{p}$ preserves the volume near $p$, then for $\delta
>0$ small, it follows from (\ref{eq2.4}) that%
\begin{equation}
PG_{L,p}^{\frac{n-4}{n-2}}=c_{n}\delta +\text{a }L^{1}\text{ function}
\label{eq2.8}
\end{equation}%
on $B_{\delta }\left( p\right) $ in distribution sense. On the other hand,
on $M\backslash \left\{ p\right\} $ using (\ref{eq1.2}) and (\ref{eq1.3}) we
have%
\begin{eqnarray}
P_{g}\left( G_{L,p}^{\frac{n-4}{n-2}}\right) &=&G_{L,p}^{\frac{n+4}{n-2}%
}P_{G_{L,p}^{\frac{4}{n-2}}g}1  \label{eq2.9} \\
&=&\frac{n-4}{2}G_{L,p}^{\frac{n+4}{n-2}}Q_{G_{L,p}^{\frac{4}{n-2}}g}  \notag
\\
&=&-\frac{n-4}{\left( n-2\right) ^{2}}G_{L,p}^{\frac{n-4}{n-2}}\left\vert
Rc_{G_{L,p}^{\frac{4}{n-2}}g}\right\vert _{g}^{2}.  \notag
\end{eqnarray}%
Here we have used the fact $R_{G_{L,p}^{\frac{4}{n-2}}g}=0$. Combine (\ref%
{eq2.8}) and (\ref{eq2.9}) we get (\ref{eq2.1}).

\section{The case dimension $n>4$\label{sec3}}

Throughout this section we will assume $\left( M,g\right) $ is a smooth
compact Riemannian manifold with dimension $n>4$.

\begin{lemma}
\label{lem3.1}Assume $n>4$, $Y\left( g\right) >0$, $u\in C^{\infty }\left(
M\right) $ such that $u\geq 0$ and $Pu\geq 0$. If for some $p\in M$, $%
u\left( p\right) =0$, then $u\equiv 0$.
\end{lemma}

\begin{proof}
By (\ref{eq2.1}) we have%
\begin{equation*}
\int_{M}G_{L,p}^{\frac{n-4}{n-2}}Pud\mu =-\frac{n-4}{\left( n-2\right) ^{2}}%
\int_{M}G_{L,p}^{\frac{n-4}{n-2}}\left\vert Rc_{G_{L,p}^{\frac{4}{n-2}%
}g}\right\vert _{g}^{2}ud\mu .
\end{equation*}%
Hence $Pu=0$ and $\left\vert Rc_{G_{L,p}^{\frac{4}{n-2}}g}\right\vert
_{g}^{2}u=0$.

If $u$ is not identically zero, then by unique continuation property we know 
$\left\{ u\neq 0\right\} $ is dense, hence $Rc_{G_{L,p}^{\frac{4}{n-2}}g}=0$%
. Since $\left( M\backslash \left\{ p\right\} ,G_{L,p}^{\frac{4}{n-2}%
}g\right) $ is asymptotically flat, it follows from relative volume
comparison theorem that $\left( M\backslash \left\{ p\right\} ,G_{L,p}^{%
\frac{4}{n-2}}g\right) $ is isometric to the standard $\mathbb{R}^{n}$. In
particular $\left( M,g\right) $ must be locally conformally flat and simply
connected compact manifold, hence it is conformal to the standard $S^{n}$ by 
\cite{K}. But in this case we have $\ker P=0$, hence $u=0$, a contradiction.
\end{proof}

\begin{remark}
\label{rmk3.1}Indeed the same argument gives us the \ following statement:
If $n>4$, $Y\left( g\right) >0$, $u\in L^{1}\left( M\right) $ such that $%
u\geq 0$ and $Pu\geq 0$ in distribution sense, for some $p\in M$, $u$ is
smooth near $p$ and $u\left( p\right) =0$, then $u\equiv 0$.
\end{remark}

A straightforward consequence of Lemma \ref{lem3.1} is the following useful
fact.

\begin{proposition}
\label{prop3.1}Assume $n>4$, $Y\left( g\right) >0$, $Q\geq 0$. If $u\in
C^{\infty }\left( M\right) $ such that $Pu\geq 0$ and $u$ is not identically
constant, then $u>0$.
\end{proposition}

\begin{proof}
If the conclusion of the proposition is false, then $u\left( p\right)
=\min_{M}u\leq 0$ for some $p$. Let $\lambda =-u\left( p\right) \geq 0$,
then $u+\lambda \geq 0$, $u\left( p\right) +\lambda =0$ and%
\begin{equation*}
P\left( u+\lambda \right) =Pu+\frac{n-4}{2}\lambda Q\geq 0.
\end{equation*}%
It follows from the Lemma \ref{lem3.1} that $u+\lambda \equiv 0$. This
contradicts with the fact $u$ is not a constant function.
\end{proof}

Proposition \ref{prop3.1} helps us determine the null space of $P$ without
information on the first eigenvalue.

\begin{corollary}
\label{cor3.1}Assume $n>4$, $Y\left( g\right) >0$, $Q\geq 0$, then%
\begin{equation*}
\ker P\subset \left\{ \text{constant functions}\right\} .
\end{equation*}%
If in addition, $Q$ is not identically zero, then $\ker P=0$ i.e. $0$ is not
an eigenvalue of $P$.
\end{corollary}

\begin{proof}
Assume $Pu=0$. If $u$ is not a constant function, then it follows from
Proposition \ref{prop3.1} that $u>0$ and $-u>0$, a contradiction.
\end{proof}

Now we ready to prove half of Theorem \ref{thm1.1}.

\begin{lemma}
\label{lem3.2}Assume $n>4$, $Y\left( g\right) >0$, $Q\geq 0$ and not
identically zero, then $\ker P=0$, moreover the Green's function $%
G_{P,p}\left( q\right) =G_{P}\left( p,q\right) >0$ for $p\neq q$.
\end{lemma}

\begin{proof}
By Corollary \ref{cor3.1}, we know $\ker P=0$. Hence for any $f\in C^{\infty
}\left( M\right) $, there exists a unique $u\in C^{\infty }\left( M\right) $
with $Pu=f$, moreover%
\begin{equation*}
u\left( p\right) =\int_{M}G_{P,p}\left( q\right) f\left( q\right) d\mu
\left( q\right) .
\end{equation*}%
If $f\geq 0$, it follows from the Proposition \ref{prop3.1} that $u\geq 0$.
Hence $G_{P,p}\geq 0$. If $G_{P,p}\left( q\right) =0$ for some $q$, since $%
PG_{P,p}=\delta _{p}\geq 0$ in distribution sense, we know from the Remark %
\ref{rmk3.1} that $G_{P,p}\equiv 0$, a contradiction. Hence $G_{P,p}\left(
q\right) >0$ for $p\neq q$.
\end{proof}

Next let us give the full argument of Theorem \ref{thm1.1}.

\begin{proof}[Proof of Theorem \protect\ref{thm1.1}]
(1)$\Rightarrow $(2): This follows from Lemma \ref{lem3.2}, (\ref{eq1.5})
and (\ref{eq1.6}).

(2)$\Rightarrow $(1): This follows from the classical Krein-Rutman theorem (%
\cite{L}). Since our case is relatively simple, we provide the argument
here. Define the integral operator $T$ as%
\begin{equation*}
Tf\left( p\right) =\int_{M}G_{P}\left( p,q\right) f\left( q\right) d\mu
\left( q\right) .
\end{equation*}%
$T$ is the inverse operator of $P$. Let%
\begin{equation*}
\alpha _{1}=\sup_{f\in L^{2}\left( M\right) \backslash \left\{ 0\right\} }%
\frac{\int_{M}Tf\cdot fd\mu }{\left\Vert f\right\Vert _{L^{2}}^{2}}>0.
\end{equation*}%
$\alpha _{1}$ is an eigenvalue of $T$. We note all eigenfunctions of $\alpha
_{1}$ does not change sign. Indeed say $T\varphi =\alpha _{1}\varphi $, $%
\int_{M}\varphi ^{2}d\mu =1$, we have%
\begin{equation*}
\int_{M}\left( \varphi _{+}^{2}+\varphi _{-}^{2}\right) d\mu =1.
\end{equation*}%
Here $\varphi _{+}=\max \left\{ \varphi ,0\right\} $, $\varphi _{-}=\max
\left\{ -\varphi ,0\right\} $. Without losing of generality, we assume $%
\varphi _{+}$ is not identically zero. Then%
\begin{eqnarray*}
\alpha _{1} &=&\int_{M}T\varphi \cdot \varphi d\mu \\
&=&\int_{M}T\varphi _{+}\cdot \varphi _{+}d\mu +\int_{M}T\varphi _{-}\cdot
\varphi _{-}d\mu -2\int_{M}T\varphi _{+}\cdot \varphi _{-}d\mu \\
&\leq &\alpha _{1}-2\int_{M}T\varphi _{+}\cdot \varphi _{-}d\mu .
\end{eqnarray*}%
Hence $\int_{M}T\varphi _{+}\cdot \varphi _{-}d\mu =0$. Since $T\varphi
_{+}>0$, we see $\varphi _{-}=0$. Hence $\varphi \geq 0$. Because $T\varphi
=\alpha _{1}\varphi $ we see $\varphi \in C^{\infty }\left( M\right) $ and $%
\varphi >0$. It follows that $\alpha _{1}$ must be a simple eigenvalue and $%
P\varphi =\alpha _{1}^{-1}\varphi $, hence%
\begin{equation*}
Q_{\varphi ^{\frac{4}{n-4}}g}=\frac{2}{n-4}P_{\varphi ^{\frac{4}{n-4}}g}1=%
\frac{2}{n-4}\varphi ^{-\frac{n+4}{n-4}}P_{g}\varphi =\frac{2}{n-4}\alpha
_{1}^{-1}\varphi ^{-\frac{8}{n-4}}>0.
\end{equation*}

(2)$\Rightarrow $(3): Assume $p_{0}\in M$ such that $G_{P,p_{0}}>0$. For $%
p\in M$, define%
\begin{equation}
\Theta \left( p\right) =\min_{q\in M\backslash \left\{ p\right\}
}G_{P}\left( p,q\right)
\end{equation}%
Then we have $\Theta \left( p_{0}\right) >0$. We note that $\Theta \left(
p\right) \neq 0$ for any $p\in M$. Otherwise, say $\Theta \left( p\right) =0$%
, then $G_{P,p}\geq 0$ and $G_{P,p}\left( q\right) =0$ for some $q\neq p$.
It follows from Remark \ref{rmk3.1} that $G_{P,p}=const$, a contradiction.
Since $M$ is connected we see $\Theta \left( p\right) >0$ for all $p$. In
another word, $G_{P}\left( p,q\right) >0$ for $p\neq q$.
\end{proof}

\begin{remark}
\label{rmk3.2}In the proof of (2)$\Rightarrow $(1), a similar argument tells
us if $\beta $ is an eigenvalue of $T$, $\beta \neq \alpha _{1}$, then $%
\left\vert \beta \right\vert <\alpha _{1}$. It follows that when $G_{P}$ is
positive, the smallest \textbf{positive} eigenvalue of $P$ must be simple
and its eigenfunction must be either strictly positive or strictly negative.
Moreover if $\lambda $ is a negative eigenvalue of $P$, then $\left\vert
\lambda \right\vert $ is strictly bigger than the smallest positive
eigenvalue.
\end{remark}

\begin{proof}[Proof of Proposition \protect\ref{prop1.1}]
By Lemma \ref{lem3.2} we know $\ker P=0$ and $G_{P}>0$. From (\ref{eq2.1})
we know%
\begin{equation*}
P\left( G_{L,p}^{\frac{n-4}{n-2}}-c_{n}G_{P,p}\right) =-\frac{n-4}{\left(
n-2\right) ^{2}}G_{L,p}^{\frac{n-4}{n-2}}\left\vert Rc_{G_{L,p}^{\frac{4}{n-2%
}}g}\right\vert _{g}^{2}\leq 0.
\end{equation*}%
Hence $G_{L,p}^{\frac{n-4}{n-2}}\leq c_{n}G_{P,p}$. If for some $q\neq p$, $%
G_{L,p}^{\frac{n-4}{n-2}}\left( q\right) =c_{n}G_{P,p}\left( q\right) $,
then $Rc_{G_{L,p}^{\frac{4}{n-2}}g}=0$, hence $\left( M,g\right) $ is
conformal diffeomorphic to the standard $S^{n}$ by the argument in the proof
of Lemma \ref{lem3.1}.
\end{proof}

\section{$3$ dimensional case\label{sec4}}

Throughout this section we assume $\left( M,g\right) $ is a smooth compact
Riemannian manifold of dimension $3$.

If $Y\left( g\right) >0$, then for $p\in M$, (\ref{eq2.1}) becomes%
\begin{equation}
P\left( G_{L,p}^{-1}\right) =-256\pi ^{2}\delta _{p}+G_{L,p}^{-1}\left\vert
Rc_{G_{L,p}^{4}g}\right\vert _{g}^{2}.  \label{eq4.1}
\end{equation}%
Note here $G_{L,p}^{-1}\in H^{2}\left( M\right) $.

\begin{lemma}
\label{lem4.1}Assume $Y\left( g\right) >0$, $u\in H^{2}\left( M\right) $
such that $u\geq 0$, $Pu\leq 0$ in distribution sense. If for some $p\in M$, 
$u\left( p\right) =0$, then either $u\equiv 0$ or $\left( M,g\right) $ is
conformal diffeomorphic to the standard $S^{3}$ and $u$ is a constant
multiple of $G_{P,p}$.
\end{lemma}

\begin{proof}
Using the fact $G_{L,p}^{-1}\in H^{2}\left( M\right) $, it follows from (\ref%
{eq4.1}) that%
\begin{equation*}
\int_{M}G_{L,p}^{-1}Pud\mu -\int_{M}G_{L,p}^{-1}\left\vert
Rc_{G_{L,p}^{4}g}\right\vert _{g}^{2}ud\mu =0.
\end{equation*}%
Note here%
\begin{equation*}
\int_{M}G_{L,p}^{-1}Pud\mu =E\left( G_{L,p}^{-1},u\right) .
\end{equation*}%
Hence $\int_{M}G_{L,p}^{-1}Pud\mu =0$ and $\int_{M}G_{L,p}^{-1}\left\vert
Rc_{G_{L,p}^{4}g}\right\vert _{g}^{2}ud\mu =0$. Hence $\left\vert
Rc_{G_{L,p}^{4}g}\right\vert _{g}^{2}u=0$. Since $Pu$ must be a measure, we
see $Pu=const\cdot \delta _{p}$. In particular $u$ is smooth on $M\backslash
\left\{ p\right\} $. If $u$ is not identically zero, it follows from unique
continuation property that the set $\left\{ u\neq 0\right\} $ is dense, and
hence $Rc_{G_{L,p}^{4}g}=0$. Same argument as in the proof of Lemma \ref%
{lem3.1} tells us $\left( M,g\right) $ must be conformal diffeomorphic to
the standard $S^{3}$, and hence $u=const\cdot G_{P,p}$.
\end{proof}

\begin{proposition}
\label{prop4.1}Assume $Y\left( g\right) >0$, $Q\geq 0$. If $u\in C^{\infty
}\left( M\right) $ such that $Pu\leq 0$ and $u$ is not identically constant,
then $u>0$.
\end{proposition}

\begin{proof}
If the conclusion of the proposition is false, then $u\left( p\right)
=\min_{M}u\leq 0$ for some $p$. Let $\lambda =-u\left( p\right) \geq 0$,
then $u+\lambda \geq 0$, $u\left( p\right) +\lambda =0$ and%
\begin{equation*}
P\left( u+\lambda \right) =Pu-\lambda Q\leq 0.
\end{equation*}%
It follows from the Lemma \ref{lem4.1} that $u+\lambda \equiv 0$. This
contradicts with the fact $u$ is not a constant function.
\end{proof}

\begin{corollary}
\label{cor4.1}Assume $Y\left( g\right) >0$, $Q\geq 0$, then $\ker P\subset
\left\{ \text{constant functions}\right\} $. If in addition, $Q$ is not
identically zero, then $\ker P=0$ i.e. $0$ is not an eigenvalue of $P$.
\end{corollary}

\begin{proof}
Assume $Pu=0$. If $u$ is not a constant function, then it follows from
Proposition \ref{prop4.1} that $u>0$ and $-u>0$, a contradiction.
\end{proof}

\begin{lemma}
\label{lem4.2}Assume $Y\left( g\right) >0$, $Q\geq 0$ and not identically
zero, then $\ker P=0$, and the Green's function $G_{P,p}\left( q\right)
=G_{P}\left( p,q\right) <0$ for $p\neq q$. Moreover if for some $p\in M$, $%
G_{P,p}\left( p\right) =0$, then $\left( M,g\right) $ is conformal
diffeomorphic to the standard $S^{3}$.
\end{lemma}

\begin{proof}
By Corollary \ref{cor4.1}, we know $\ker P=0$. Hence for any $f\in C^{\infty
}\left( M\right) $, there exists a unique $u\in C^{\infty }\left( M\right) $
with $Pu=f$, moreover%
\begin{equation*}
u\left( p\right) =\int_{M}G_{P,p}\left( q\right) f\left( q\right) d\mu
\left( q\right) .
\end{equation*}%
If $f\leq 0$, it follows from the Proposition \ref{prop4.1} that $u\geq 0$.
Hence $G_{P,p}\leq 0$. If $G_{P,p}\left( q\right) =0$ for some $q$, since $%
PG_{P,p}=\delta _{p}\geq 0$, it follows from Lemma \ref{lem4.1} that $\left(
M,g\right) $ must be conformal diffeomorphic to the standard $S^{3}$ and $%
G_{P,p}$ is a constant multiple of $G_{P,q}$, this implies $p=q$. Hence $%
G_{P,p}<0$ on $M\backslash \left\{ p\right\} $.
\end{proof}

Now we are ready to prove Theorem \ref{thm1.2}.

\begin{proof}[Proof of Theorem \protect\ref{thm1.2}]
(1)$\Rightarrow $(2): This follows from Lemma \ref{lem4.2} and (\ref{eq1.5}%
), (\ref{eq1.6}).

(2)$\Rightarrow $(1): This follows from Krein-Rutman theorem, or one may use
the argument in the proof of Theorem \ref{thm1.1}. We also remark it follows
that the largest \textbf{negative }eigenvalue of $P$ must be simple and its
eigenfunction must be strictly positive or strictly negative. Moreover if $%
\lambda $ is a positive eigenvalue of $P$, then $\lambda $ is strictly
bigger than the absolute value of the largest negative eigenvalue.

(3)$\Rightarrow $(2): We can assume $\left( M,g\right) $ is not conformal
diffeomorphic to the standard $S^{3}$. For any $p\in M$, we let%
\begin{equation*}
\Theta \left( p\right) =\max_{M}G_{P,p}.
\end{equation*}%
Then it follows from Lemma \ref{lem4.1} that $\Theta \left( p\right) \neq 0$
for any $p\in M$. Since $\Theta \left( p_{0}\right) <0$ for some $p_{0}\in M$%
, we see $\Theta \left( p\right) <0$ for all $p\in M$. In another word, $%
G_{P}<0$.
\end{proof}

With all the above analysis, we can easily deduce Proposition \ref{prop1.2}.

\begin{proof}[Proof of Proposition \protect\ref{prop1.2}]
Under the assumption of Proposition \ref{prop1.2}, it follows from Lemma \ref%
{lem4.2} that $\ker P=0$ and $G_{P}\left( p,q\right) <0$ for $p\neq q$. From
(\ref{eq4.1}) we see%
\begin{equation*}
P\left( G_{L,p}^{-1}+256\pi ^{2}G_{P,p}\right) =G_{L,p}^{-1}\left\vert
Rc_{G_{L,p}^{4}g}\right\vert _{g}^{2}\geq 0.
\end{equation*}%
Hence $G_{L,p}^{-1}+256\pi ^{2}G_{P,p}\leq 0$. If it achieves $0$ somewhere,
then $Rc_{G_{L,p}^{4}g}=0$ and hence $\left( M,g\right) $ is conformal
diffeomorphic to the standard $S^{3}$.
\end{proof}

At last we want to point out based on Proposition \ref{prop1.2}, using the
arguments in \cite{HY3} we have the following statement: Let%
\begin{equation*}
\mathcal{M}=\left\{ g:%
\begin{tabular}{l}
$g$ is a smooth metric with $Y\left( g\right) >0$ and there exists \\ 
a positive smooth function $\rho $ such that $Q_{\rho ^{2}g}>0$%
\end{tabular}%
\right\}
\end{equation*}%
be endowed with $C^{\infty }$ topology. Then

\begin{enumerate}
\item For every $g\in \mathcal{M}$, there exists $\rho \in C^{\infty }\left(
M\right) $, $\rho >0$ such that $Q_{\rho ^{2}g}=1$. Moreover as long as $%
\left( M,g\right) $ is not conformal diffeomorphic to the standard $S^{3}$,
the set%
\begin{equation*}
\left\{ \rho \in C^{\infty }\left( M\right) :\rho >0,Q_{\rho ^{2}g}=1\right\}
\end{equation*}%
is compact in $C^{\infty }$ topology.

\item Let $\mathcal{N}$ be a path connected component of $\mathcal{M}$. If
there is a metric in $\mathcal{N}$ satisfying condition NN, then every
metric in $\mathcal{N}$ satisfies condition NN. Hence as long as the metric
is not conformal to the standard $S^{3}$, it satisfies condition $P$. As a
consequence, for any metric in $\mathcal{N}$,%
\begin{equation*}
\inf \left\{ E\left( u\right) \left\Vert u^{-1}\right\Vert _{L^{6}\left(
M\right) }^{2}:u\in H^{2}\left( M\right) ,u>0\right\} >-\infty
\end{equation*}%
and is always achieved.
\end{enumerate}

We omit the details here.

\section{$4$ dimension case revisited\label{sec5}}

Throughout this section we will assume $\left( M,g\right) $ is a smooth
compact Riemannian manifold of dimension $4$. In this dimension the $Q$
curvature is written as%
\begin{equation}
Q=-\frac{1}{6}\Delta R-\frac{1}{2}\left\vert Rc\right\vert ^{2}+\frac{1}{6}%
R^{2}.  \label{eq5.1}
\end{equation}%
The Paneitz operator can be written as%
\begin{equation}
P\varphi =\Delta ^{2}\varphi +2\func{div}\left( Rc\left( \nabla \varphi
,e_{i}\right) e_{i}\right) -\frac{2}{3}\func{div}\left( R\nabla \varphi
\right) .  \label{eq5.2}
\end{equation}%
Here $e_{1},e_{2},e_{3},e_{4}$ is a local orthonormal frame with respect to $%
g$. $P$ satisfies%
\begin{equation}
P_{e^{2w}g}\varphi =e^{-4w}P_{g}\varphi  \label{eq5.3}
\end{equation}%
for any smooth function $w$. The $Q$ curvature transforms as%
\begin{equation}
Q_{e^{2w}g}=e^{-4w}\left( P_{g}w+Q_{g}\right) .  \label{eq5.4}
\end{equation}%
In the spirit of Proposition \ref{prop2.1}, we have

\begin{proposition}
\label{prop5.1}Assume $\left( M,g\right) $ is a $4$ dimensional smooth
compact Riemannian manifold with $Y\left( g\right) >0$, $p\in M$, then we
have $\left\vert Rc_{G_{L,p}^{2}g}\right\vert _{g}^{2}\in L^{1}\left(
M\right) $ and%
\begin{equation}
P\left( \log G_{L,p}\right) =16\pi ^{2}\delta _{p}-\frac{1}{2}\left\vert
Rc_{G_{L,p}^{2}g}\right\vert _{g}^{2}-Q  \label{eq5.5}
\end{equation}%
in distribution sense. Here $G_{L,p}$ is the Green's function of conformal
Laplacian operator $L=-6\Delta +R$ with pole at $p$.
\end{proposition}

\begin{proof}
If $\rho $ is a positive smooth function on $M$, $\widetilde{g}=\rho ^{2}g$,
then%
\begin{equation}
\left\vert Rc_{G_{\widetilde{L},p}^{2}\widetilde{g}}\right\vert _{\widetilde{%
g}}^{2}d\widetilde{\mu }=\left\vert Rc_{G_{L,p}^{2}g}\right\vert
_{g}^{2}d\mu .  \label{eq5.6}
\end{equation}%
Hence to show $\left\vert Rc_{G_{L,p}^{2}g}\right\vert _{g}^{2}\in
L^{1}\left( M\right) $, in view of the existence of conformal normal
coordinate, we can assume $\exp _{p}$ preserves volume near $p$. Let $%
x_{1},x_{2},x_{3},x_{4}$ be normal coordinate at $p$, $r=\left\vert
x\right\vert $, then (see \cite{LP})%
\begin{equation}
G_{L,p}=\frac{1}{24\pi ^{2}}\frac{1}{r^{2}}\left( 1+O^{\left( 4\right)
}\left( r^{2}\right) \right) .  \label{eq5.7}
\end{equation}%
Using%
\begin{eqnarray}
Rc_{G_{L,p}^{2}g} &=&Rc-2D^{2}\log G_{L,p}+2d\log G_{L,p}\otimes d\log
G_{L,p}  \label{eq5.8} \\
&&-\left( \Delta \log G_{L,p}+2\left\vert \nabla \log G_{L,p}\right\vert
^{2}\right) g,  \notag
\end{eqnarray}%
we see $\left\vert Rc_{G_{L,p}^{2}g}\right\vert _{g}=O\left( 1\right) $,
hence $\left\vert Rc_{G_{L,p}^{2}g}\right\vert _{g}^{2}\in L^{1}\left(
M\right) $.

On the other hand, (\ref{eq5.5}) means%
\begin{equation}
\int_{M}\log G_{L,p}\cdot P\varphi d\mu =16\pi ^{2}\varphi \left( p\right) -%
\frac{1}{2}\int_{M}\left\vert Rc_{G_{L,p}^{2}g}\right\vert _{g}^{2}\varphi
d\mu -\int_{M}Q\varphi d\mu .  \label{eq5.9}
\end{equation}%
Careful check shows (\ref{eq5.9}) is conformally invariant. Hence we can
assume $\exp _{p}$ preserves volume near $p$. It follows from (\ref{eq5.7})
that on $B_{\delta }\left( p\right) $ for $\delta >0$ small,%
\begin{equation}
P\left( \log G_{L,p}\right) =16\pi ^{2}\delta _{p}+\text{a }L^{1}\text{
function}  \label{eq5.10}
\end{equation}%
in distribution sense. On $M\backslash \left\{ p\right\} $, we have%
\begin{equation}
P\left( \log G_{L,p}\right) =G_{L,p}^{4}Q_{G_{L,p}^{2}g}-Q=-\frac{1}{2}%
\left\vert Rc_{G_{L,p}^{2}g}\right\vert _{g}^{2}-Q.  \label{eq5.11}
\end{equation}%
(\ref{eq5.5}) follows.
\end{proof}

By integrating (\ref{eq5.5}) on $M$ and observing that%
\begin{equation*}
\left\vert Rc_{G_{L,p}^{2}g}\right\vert _{g}^{2}d\mu _{g}=\left\vert
Rc_{G_{L,p}^{2}g}\right\vert _{G_{L,p}^{2}g}^{2}d\mu _{G_{L,p}^{2}g}
\end{equation*}%
we immediately get

\begin{corollary}
\label{cor5.1}Assume $Y\left( g\right) >0$, then for any $p\in M$,%
\begin{equation}
\int_{M}Qd\mu +\frac{1}{2}\int_{M}\left\vert Rc_{G_{L,p}^{2}g}\right\vert
_{G_{L,p}^{2}g}^{2}d\mu _{G_{L,p}^{2}g}=16\pi ^{2}.  \label{eq5.12}
\end{equation}%
In particular, $\int_{M}Qd\mu \leq 16\pi ^{2}$ and equality holds if and
only if $\left( M,g\right) $ is conformal diffeomorphic to the standard $%
S^{4}$.
\end{corollary}

\section{Positive mass theorem for Paneitz operator revisited\label{sec6}}

Throughout this section we will assume $\left( M,g\right) $ is a smooth
compact Riemannian manifold with dimension $n>4$.

In \cite{HuR}, for locally conformally flat manifold with $Y\left( g\right)
>0$ and positive Green's function $G_{P}$, a positive mass theorem for
Paneitz operator was proved by a nice calculation. Note that this result
plays similar role for $Q$ curvature equation as the classical positive mass
theorem for the Yamabe problem (\cite{LP}). It was observed that similar
calculation works for $n=5,6,7$ in \cite{GM} and for $n=3$ in \cite{HY3}.
Since the case $n=3$ can be covered by Lemma \ref{lem4.1}, we concentrate on
the case $n>4$. The main aim of this section is to show the positive mass
theorem for Paneitz operator follows from the formula (\ref{eq2.1}).

\begin{lemma}
\label{lem6.1}Assume $n>4$, $Y\left( g\right) >0$, $\ker P=0$. Let $%
x_{1},\cdots ,x_{n}$ be a coordinate near $p$ with $x_{i}\left( p\right) =0$%
, $r=\left\vert x\right\vert $. If either $M$ is conformally flat near $p$
or $n=5,6,7$, then%
\begin{equation}
c_{n}G_{P,p}-G_{L,p}^{\frac{n-4}{n-2}}=\text{const}+O^{\left( 4\right)
}\left( r\right) .  \label{eq6.1}
\end{equation}%
Here $c_{n}$ is the constant given by (\ref{eq1.13}).
\end{lemma}

\begin{proof}
First we observe that if $\rho $ is a positive smooth function on $M$, $%
\widetilde{g}=\rho ^{\frac{4}{n-4}}g$, then%
\begin{equation}
c_{n}G_{\widetilde{P},p}-G_{\widetilde{L},p}^{\frac{n-4}{n-2}}=\rho \left(
p\right) ^{-1}\rho ^{-1}\left( c_{n}G_{P,p}-G_{L,p}^{\frac{n-4}{n-2}}\right)
.  \label{eq6.2}
\end{equation}%
Hence we only need to verify (\ref{eq6.1}) for a conformal metric.

For the case $M$ is conformally flat near $p$, by a conformal change of
metric, we can assume $g$ is Euclidean near $p$. Then under the normal
coordinate at $p$ we have%
\begin{equation}
G_{P,p}=\frac{1}{2n\left( n-2\right) \left( n-4\right) \omega _{n}}\left(
r^{4-n}+A+O^{\left( 4\right) }\left( r\right) \right) .  \label{eq6.3}
\end{equation}%
Here $\omega _{n}$ is the volume of unit ball in $\mathbb{R}^{n}$ and $A$ is
a constant. People usually call $A$ as the mass of Paneitz operator. The
Green's function of conformal Laplacian%
\begin{equation}
G_{L,p}=\frac{1}{4n\left( n-1\right) \omega _{n}}\left( r^{2-n}+O^{\left(
4\right) }\left( r^{-1}\right) \right) .  \label{eq6.4}
\end{equation}%
It is worth pointing out one has better estimate for the Green's function
than the one in (\ref{eq6.3}) and (\ref{eq6.4}), but the formula we wrote
above also works for $n=5,6,7$ without locally conformally flat assumption.
More precisely, for $n=5,6,7$, under the conformal normal coordinate, (\ref%
{eq6.3}) and (\ref{eq6.4}) remain true (see \cite{GM,LP}). It follows that%
\begin{equation}
c_{n}G_{P,p}-G_{L,p}^{\frac{n-4}{n-2}}=\left( 4n\left( n-1\right) \omega
_{n}\right) ^{-\frac{n-4}{n-2}}A+O^{\left( 4\right) }\left( r\right) .
\label{eq6.5}
\end{equation}
\end{proof}

To continue, we note that under the assumption of Lemma \ref{lem6.1}, by (%
\ref{eq2.1}) we have%
\begin{equation}
P\left( c_{n}G_{P,p}-G_{L,p}^{\frac{n-4}{n-2}}\right) =\frac{n-4}{\left(
n-2\right) ^{2}}G_{L,p}^{\frac{n-4}{n-2}}\left\vert Rc_{G_{L,p}^{\frac{4}{n-2%
}}g}\right\vert _{g}^{2},  \label{eq6.6}
\end{equation}%
hence%
\begin{equation}
G_{L,p}^{\frac{n-4}{n-2}}\left\vert Rc_{G_{L,p}^{\frac{4}{n-2}}g}\right\vert
_{g}^{2}=O\left( r^{-3}\right)  \label{eq6.7}
\end{equation}%
and%
\begin{equation}
\left( 4n\left( n-1\right) \omega _{n}\right) ^{-\frac{n-4}{n-2}}A=\frac{n-4%
}{\left( n-2\right) ^{2}}\int_{M}G_{P,p}G_{L,p}^{\frac{n-4}{n-2}}\left\vert
Rc_{G_{L,p}^{\frac{4}{n-2}}g}\right\vert _{g}^{2}d\mu .  \label{eq6.8}
\end{equation}

If in addition we know the Green's function $G_{P,p}>0$, then it follows
from (\ref{eq6.8}) that $A\geq 0$, moreover $A=0$ if and only if $\left(
M,g\right) $ is conformal equivalent to the standard $S^{n}$. This proves
the positive mass theorem for Paneitz operator.

\end{document}